\documentclass[12pt,reqno]{amsart}
\usepackage{diagrams}
\usepackage{amssymb}

\numberwithin{equation}{section}

\newtheorem{thm}{Theorem}[section]
\newtheorem{cor}[thm]{Corollary}
\newtheorem{lem}[thm]{Lemma}

\newtheorem{prop}[thm]{Proposition}

\theoremstyle{definition}

\newtheorem{rem}[thm]{Remark}
\newtheorem{note}[thm]{Notation}
\newtheorem{se}[thm]{Setup}

\DeclareMathOperator{\NL}{\mathrm{NL}}

\DeclareMathOperator{\Ima}{\mathrm{Im}}
\DeclareMathOperator{\p3}{\mathbb{P}^3}

\DeclareMathOperator{\pr}{\mathrm{pr}}
\DeclareMathOperator{\Spec}{\mathrm{Spec}}

\DeclareMathOperator{\N}{\mathcal{N}}
\DeclareMathOperator{\T}{\mathcal{T}}
\DeclareMathOperator{\I}{\mathcal{I}}
\DeclareMathOperator{\mo}{\mathcal{O}}

\newcommand{\mb}[1]{\mathbb{#1}}
\newcommand{\mc}[1]{\mathcal{#1}}
\newcommand{\mr}[1]{\mathrm{#1}}
\newcommand{\ov}[1]{\overline{#1}}

\begin{document}

\title[Hodge locus and Brill-Noether type locus]{Hodge locus and Brill-Noether type 
locus}

\author[I. Biswas]{Indranil Biswas}

\address{School of Mathematics, Tata Institute of Fundamental
Research, Homi Bhabha Road, Mumbai 400005, India}

\email{indranil@math.tifr.res.in}

\author[A. Dan]{Ananyo Dan}

\address{School of Mathematics, Tata Institute of Fundamental
Research, Homi Bhabha Road, Mumbai 400005, India}

\email{dan@math.tifr.res.in}



\date{}

\begin{abstract}
Given a family $\pi\,:\,\mc{X} \,\longrightarrow\, B$ of smooth projective varieties, a
closed fiber $\mc{X}_o$ and an invertible sheaf $\mc{L}$ on $\mc{X}_o$, we compare 
the Hodge locus in $B$ corresponding to the Hodge class $c_1(\mc{L})$ with 
the locus of points $b\,\in\, B$ such that $\mc{L}$ deforms to an invertible sheaf 
$\mc{L}_b$ on $\mc{X}_b$ with at least $h^0(\mc{L})$--dimensional space of global
sections (it is a Brill-Noether type locus associated to $\mc{L}$).
We finally give an application by comparing the Brill-Noether locus to a 
family of curves on a surface passing through a fixed set of points.
\end{abstract}

\maketitle

\section{Introduction}

The base field $k$ is always assumed to be algebraically closed of characteristic zero. 
Consider a family $\pi\,:\,\mc{X} \,\longrightarrow\, B$ of smooth, projective 
varieties with a reference point $o \,\in \, B$ and a Hodge class $\gamma \in 
H^{1,1}(\mc{X}_o,\, \mb{Z})$, where $\mc{X}_b\,:=\,\pi^{-1}(b)$. The Hodge locus 
$\NL(\gamma)\, \subset\, B$ corresponding to $\gamma$ is the space of all $b \,\in\, 
B$ such that $\gamma$ deforms to a Hodge class on $\mc{X}_b$. By Lefschetz 
($1,1$)-theorem, $\gamma\,=\,c_1(\mc{L})$ for some invertible sheaf $\mc{L}$ on 
$\mc{X}_o$. We compare $\NL(\gamma)$ with a Brill-Noether type locus 
$\mc{B}_{\mc{L}}$ associated to $\mc{L}$. More precisely, when $h^0(\mc{L})\,>\,1$,
we define $\mc{B}_{\mc{L}}$ to be the sub-locus of $\NL(\gamma)$ consisting of all points 
$b \,\in\, \NL(\gamma)$ for which $\mc{L}$ deforms to an invertible sheaf $\mc{L}_b$ 
on $\mc{X}_b$ satisfying $h^0(\mc{L}_b) \,\ge\, h^0(\mc{L})$. In some sense, 
$\mc{B}_{\mc{L}}$ consists of those points of $\NL(\gamma)$ for which the entire 
linear system $|\mc{L}|$ deforms. The study of these two loci is related to the 
following classical question:

Given a family $\pi$ as above and a closed fiber $\mc{X}_o$, classify 
effective divisors $D \,\subset\, \mc{X}_o$ satisfying the property: for any 
infinitesimal deformation $\mc{X}_t$ of $\mc{X}_o$ corresponding to $t \,\in\, T_oB$, 
the Hodge class $[D]$ corresponding to $D$ lifts to a Hodge class on $\mc{X}_t$ if 
and only if $D$ lifts to an effective Cartier divisor on $\mc{X}_t$?

This question is still wide open. Bloch proved in \cite{b1}
that semi-regular Cartier divisors satisfy this property. But semi-regularity
is a very strong condition and there are several examples of Cartier divisors that are not 
semi-regular but satisfy this property. In this article, we address the question in terms of 
the Brill-Noether locus associated to $\mc{L}\,=\,\mo_{\mc{X}_o}(D)$. In particular, we prove 
that if the Hodge locus corresponding to $[D]$ coincides with the Brill-Noether
locus $\mc{B}_{\mc{L}}$ for $\mc{L}\,=\,\mo_{\mc{X}_o}(D)$, then $D$ satisfies 
the property in the question (see Theorem \ref{br12}). Although we do not prove,
but one can observe from the text that in most cases, this condition will
in fact exhaustively classify all such divisors.

The other motivation is to study deformation of linear 
systems. This $\mc{B}_{\mc{L}}$ is the correct object to consider for this 
purpose. One could naively define, 
$\mc{B}_{\mc{L}}$ to be the locus of points $b \,\in\, B$ such that every element of the 
linear system $|\mc{L}|$ deforms to an effective divisor on $\mc{X}_b$. 
But this will give us the wrong infinitesimal information, meaning the 
infinitesimal definition of $\mc{B}_{\mc{L}}$ will not agree with its global 
definition. This can be explained using relative Hilbert schemes. More precisely, 
one expects $T_o\mc{B}_{\mc{L}}$ to consist of those tangent vectors $t \,\in\, 
T_o\NL(\gamma)$ for which every effective divisor of $|\mc{L}|$ lifts to an 
effective Cartier divisor on $\mc{X}_t$, where $\mc{X}_t$ is the infinitesimal 
deformation of $\mc{X}_o$ corresponding to $t$. But, quite often, this is not going to be the 
actual tangent space at $o$ to $\mc{B}_{\mc{L}}$. In most cases with
$h^0(\mc{L})\,>\,1$, for any $b \,\in \,\NL(\gamma)$ there exists a deformation of 
$\mc{L}$ to an invertible sheaf $\mc{L}_b$ on $\mc{X}_b$ satisfying 
$h^0(\mc{L}_b)\,>\,1$. In such cases, it is not hard to show that the naive definition 
of $\mc{B}_{\mc{L}}$ will be equal to $\NL(\gamma)$. It is possible that the dimension 
of the linear system $|\mc{L}|$ jumps, i.e., there is an open neighborhood $U\,\subset\, 
\NL(\gamma)$ of $o$ such that for all $u \,\in\, U\backslash\{o\}$ and deformation 
$\mc{L}_u$ of $\mc{L}$ to an invertible sheaf $\mc{L}_u$ on $\mc{X}_u$,
$$
h^0(\mc{L})\,>\, h^0(\mc{L}_u)\, .
$$
But in this case one observes that $T_o\mc{B}_{\mc{L}} 
\,\subsetneqq\, T_o\NL(\gamma)$ even when $\NL(\gamma)$ is smooth at $o$. This would 
mean the dimension of the naive definition of $\mc{B}_{\mc{L}}$ is strictly greater 
than $\dim T_o\mc{B}_{\mc{L}}$, which is not possible (see example
in Section \ref{bsec4}). To 
resolve such ambiguity we use the Brill-Noether type definition of $\mc{B}_{\mc{L}}$.

We now discuss the approach taken in this article.
For any $D \,\in\, |\mc{L}|$, one can define a class $\{D\}
\,\in\, H^0(\mc{H}^1_D(\Omega^1_{\mc{X}_o}))$; this is a classical construction.
In fact, $c_1(\mc{L})$ is the image 
of $\{D\}$ under the natural homomorphism from $H^0(\mc{H}^1_D(\Omega^1_{\mc{X}_o}))
\,\cong\, H^1_D(\Omega_{\mc{X}_o}^1)$ to $H^1(\Omega^1_{\mc{X}_o})$. 
The tangent space $T_o\NL(\gamma)$ is given using the cup-product map
$$\cup c_1(\mc{L})\,:\,H^1(\T_{\mc{X}_o}) \,
\longrightarrow\, H^2(\mo_{\mc{X}_o})\, .$$
In particular, one uses the Kodaira-Spencer map
$\rho_\pi\,:\,T_oB \,\longrightarrow\, H^1(\T_{\mc{X}_o})$
associated to $\pi$. Then $t \,\in\, T_o\NL(\gamma)$ if and only if
$\rho_\pi(t) \cup c_1(\mc{L}) \,\longmapsto\, 0$.
Analogous to the cup-product map, one can define an inner
multiplication \[\lrcorner \{D\}\,:\,H^1(\T_{\mc{X}_o}) \,\longrightarrow\,
H^2_D(\mo_{\mc{X}_o})\, .\]
 We prove that $t \,\in\, \T_o\mc{B}_{\mc{L}}$ if and only if $\rho_{\pi}(t)
\lrcorner \{D\}\,=\,0$ (Proposition \ref{br3}). As $\mc{L}$ deforms along $\NL(\gamma)$, it is possible
 that the dimension of the space of its global sections drops, sometimes to zero. In such
cases $\mc{B}_{\mc{L}}$ and $\NL(\gamma)$ differ, and so do their tangent spaces. 
 Using Lefschetz ($1,1$)-theorem and deformation theory, there exists an invertible sheaf
$\widetilde{\mc{L}}$ on $\pi^{-1}(\NL(\gamma))$ satisfying
$\widetilde{\mc{L}}|_{\mc{X}_o} \,\cong\, \mc{L}$. We first prove:

\begin{thm}[{Theorem \ref{br12}}]\label{br13}
For $\gamma\,=\,c_1(\mc{L}) \,\in\, H^{1,1}({\mc{X}_o},\mb{Z})$, if
$h^0(\widetilde{\mc{L}}|_{\mc{X}_b})\,=\,h^0(\widetilde{\mc{L}}|_{\mc{X}_o})$ for all $b
\,\in \,B$, then $\NL(\gamma)\,=\,\mc{B}_{\mc{L}}$ and
$T_{o}\NL(\gamma)\,=\,T_{o}\mc{B}_{\mc{L}}$.
\end{thm}

We then produce an example of a family $\pi$ as above and an invertible sheaf 
$\mc{L}$ on $\mc{X}_o$ for which the hypothesis of Theorem \ref{br13} fails. For 
this example, $\mc{B}_{\mc{L}}$ is \emph{properly} contained in $\NL(\gamma)$
and so is their respective tangent spaces (see Theorem \ref{br1} and Theorem 
\ref{br14}). The point to note is that the failure of the hypothesis of 
Theorem \ref{br13} does not a-priori imply that $T_o\mc{B}_{\mc{L}} \,\not=\, 
T_o\NL(\gamma)$. It is possible for all \emph{first order} infinitesimal deformation 
$\mc{X}_t$ of $\mc{X}_o$ corresponding to $t \,\in \,T_o\NL(\gamma)$, one has 
$h^0(\widetilde{\mc{L}}|_{\mc{X}_t}) \,\ge\, h^0(\mc{L})$. Of course, there exists 
higher order infinitesimal deformations $\mc{X}_{t_n}$ of $\mc{X}_o$, along 
$\NL(\gamma)$ with $h^0(\widetilde{\mc{L}}|_{\mc{X}_{t_n}})\,<\, h^0(\mc{L})$.
 
We finally produce a family $\pi\,:\,\mc{X}\,\longrightarrow\, B$ such that
$T_o\mc{B}_{\mc{L}} \,\not=\, 
T_o\NL(\gamma)$. To produce such a family we start with a smooth, projective variety 
$X$ and an invertible sheaf $\mc{L}$ on $X$ such that the set of base points $B$ of 
$\mc{L}$ is zero dimensional. Choose a point $p \,\in\, B$, and define $B_p
\,:=\,B \backslash 
\{p\}$. We produce a flat family $\pi\,:\,\mc{X}\,\longrightarrow\, X\backslash B_p$
such that for all 
$q \,\in\, X\backslash B_p$, the fiber $\pi^{-1}(q)$ is the blow up of $X$ at
$B_p \cup q$. Denote 
by $E_q$ the exceptional divisor. There exists an invertible sheaf $\mc{M}$ on 
$\mc{X}$ such that $\mc{M}_q\,:=\,\mc{M}|_{\mc{X}_q}\,=\,\mc{L}(-E_q)$
for all $q \,\in \,X\backslash B_p$. We prove that

\begin{thm}[{Theorem \ref{br1} and Theorem \ref{br14}}]\label{br17}
Let $\gamma\,=\,c_1(\mc{M}_p)$. Then, $\dim \mc{B}_{\mc{M}_p}\,<\,\dim \NL(\gamma)$
and $T_p\NL(\gamma)\,\not=\, T_p\mc{B}_{\mc{M}_p}$.
\end{thm}

Finally, given a smooth projective surface $X$, an invertible sheaf $\mc{L}$ and a 
positive integer $n$, we study the locus of points in $X$ such that there exists a 
family of smooth projective curves in the linear system $|\mc{L}|$ of dimension at 
least $n$, passing through these points. We prove that the locus of such points is 
a Brill-Noether type locus (see Theorem \ref{br18}).

\section{Preliminaries}

We recall some basics on local cohomology groups.

Let $X$ be a topological space, $Y \,\subset\, X$ a closed subspace and $\mc{F}$ a
sheaf of abelian groups on $X$. Let $\Gamma_Y(X,\,\mc{F})$ denote the \emph{group of
sections of} $\mc{F}$ \emph{with support on} $Y$; it is also the subgroup of
$\Gamma(X,\,\mc{F})$ consisting of all sections whose support is contained in $Y$. Now,
$\Gamma_Y(X,\, -)$ is a left exact functor from the category of abelian sheaves on $X$
to abelian groups. We denote the right derived functor of $\Gamma_Y(X,\, -)$ by
$H^i_Y(X,\,-)$. They are the \emph{cohomology groups of } $X$ \emph{ with
support in } $Y$ and coefficients in a given sheaf.

For $\mc{F}$ as above, let $\underline{\Gamma}_Y(\mc{F})$ be the sheaf which associates
to an open subset $U$ the abelian group $\Gamma_{Y \cap U}(U,\,\mc{F}|_U)$.
Denote by $\mc{H}^i_Y(\mc{F})$ the associated right derived functor.

Using \cite[Proposition $1.2$]{grh} one notices that $\mc{H}^i_Y(\mc{F})$ is in fact
the sheaf associated to the presheaf which associates the abelian
group $H^i_{Y \cap U}(U,\,\mc{F}|_U)$ to an open subset $U \,\subset\, X$.

\begin{lem}[{\cite[Corollary $1.1.9$]{grh}}]\label{ph14}
Let $\mc{F}$ be a quasi-coherent sheaf on $X$. Let $U\,:=\, X-Y$ be the complement
with $j\,:\,U \,\hookrightarrow\, X$ the inclusion. There is a long exact sequence
\[
0 \,\longrightarrow\, H^0_Y(X,\,\mc{F})\,\longrightarrow\, H^0(X,\,\mc{F})\,\longrightarrow\,
 H^0(U,\,\mc{F}|_U) \,\longrightarrow\, H^1_Y(X,\,\mc{F})
\]
\[
\longrightarrow\, H^1(X,\,\mc{F})\,\longrightarrow\, H^1(U,\,\mc{F}|_U)
\,\longrightarrow\,H^2_Y(X,\mc{F})\,\longrightarrow\, \cdots \, .\]
Similarly, there is a short exact sequence
 \[0 \,\longrightarrow\, \mc{H}^0_Y(X,\,\mc{F})
\,\longrightarrow\, \mc{H}^0(X,\,\mc{F})\,\longrightarrow\, \mc{H}^0(U,\,\mc{F}|_U)
\,\stackrel{\delta}{\longrightarrow}\, \mc{H}^1_Y(X,\,\mc{F}) \,\longrightarrow\, 0\]
and $\mc{H}^{i+1}_Y(\mc{F}) \,\cong\, R^i j_*(\mc{F}|_U)$ for all $i\,>\,0$.
\end{lem}

\begin{prop}\label{ph15}
Let $X$ be a scheme, $Z$ a local complete intersection subscheme in $X$ and $\mc{F}$ a sheaf of abelian groups on $X$.
Then the spectral sequence with terms $E_2^{p,q}\,=\,H^p(X,\,\mc{H}_Z^q(X,\,
\mc{F}))$ converges to $H_Z^{p+q}(X,\,\mc{F})$. Furthermore, if $\mc{F}$ is a locally
free $\mo_X$--module, then $H^{p+q}_Z(X,\,\mc{F})\,\cong\, H^p(X,\,\mc{H}^q_Z(X,\,
\mc{F}))$, where $q$ is the codimension of $Z$ in $X$ and $p \,\ge\, 0$.
\end{prop}

\begin{proof}
The first statement is proven in \cite[Proposition $1.4$]{grh}.

Assume that $\mc{F}$ is locally free. We will show that $\mc{H}^k_Z(X,\,\mc{F})
\,=\,0$ for $k \,\not=\, q$.

Since $Z$ is a local complete intersection subscheme in $X$, there
exists an affine open covering $\{U_i\}$ of $X$ such that for each $i$ satisfying 
$Z \cap U_i \,\not=\, \emptyset$, the $\mo_X(U_i)$--module $\I_Z(U_i)$ is generated by a
$\mo_X(U_i)$--regular sequence of length $q$. In the terminology of
\cite[Ex. III.$3.4$]{R1}, this is equivalent to the assertion that
$\mbox{depth}_{\I_Z(U_i)}(\mo_X(U_i))\,=\,q$. By taking a refinement of the covering 
$\{U_i\}$ if necessary, we can also assume that $$\mc{F}|_{U_i}\,\cong\, 
\bigoplus_{i=1}^{\mr{rk}(\mc{F})}\mo_{U_i}\, .$$
This means that for all $i$ satisfying $Z \cap U_i \,\not=\, \emptyset$, we have 
$\mbox{depth}_{\I_Z(U_i)}(\mc{F}(U_i))\,=\,q$. Using \cite[Ex. III.$3.3 \mbox{ and } 
3.4$]{R1}, it follows that $H^k_{Z \cap U_i}(U_i,\,\mc{F}|_{U_i})\,=\,0$ for all $k\,<\,q$.

Now, $H^k(U_i\backslash Z,\,\mc{F})\,\cong\, H^{k+1}_{Z\cap U_i}(U_i,\,\mc{F}|_{U_i})$
for all $k\,\ge\, 1$ (see \cite[Proposition $2.2$]{grh}). By construction, we have 
$\I_Z(U_i)\,=\,(f_1^{(i)},\cdots ,f_q^{(i)})$ if $Z \cap U_i \,\not=\, \emptyset$.
Hence, any such compelment $U_i\setminus Z$ can be covered by $q$ open affine sets, 
$V_j^{(i)}\,:=\,D(f_j^{(i)})$ for $j=1,\cdots ,q$. Then, \cite[III. Ex. $4.8$]{R1}
implies that $H^k(U_i\backslash Z,\,\mc{F})\,=\,0$ for $k \,\ge\, q$, and hence
$H^{k}_{Z \cap U_i}(U_i,\,\mc{F}|_{U_i})\,=\,0$ for all $k \,\ge\, q+1$.

As $\mc{H}^k_Z(X,\,\mc{F})$ is supported on $Z$, this means that 
$\mc{H}^k_Z(X,\,\mc{F})\,=\,0$ for $k \,\not=\, q$. Since 
$$E_2^{p,q}\,=\,H^p(X,\,\mc{H}_Z^q(X,\,\mc{F}))\,\Rightarrow\, H^{p+q}_Z(X,\,\mc{F})\, ,$$
we conclude that $H^{p+q}_Z(X,\,\mc{F}) \,\cong\, H^p(X,\,\mc{H}^q_Z(\mc{F}))$.
This completes the proof.
\end{proof}

\section{Relative Brill-Noether type locus}

Let $X$ be a smooth projective surface in $\p3$, and let $\mc{L}$ be an 
invertible sheaf on $X$ with $h^0(\mc{L})\,>\,1$. Assume that there is a
reduced divisor on $X$ lying in the complete linear system $|\mc{L}|$. Let
$\{U_i\}_{i \in I}$ be a Zariski open affine covering of $X$ such that $D \cap U_i$ 
is defined by a single equation, say $f_i\,=\,0$ with $f_i \,\in\, \Gamma(U_i,\,\mo_X)$.
Denote by $V_i$ the open affine set $U_i\backslash \{f_i=0\}$.

To describe the cohomology class of $D$ in $H^2(X,\,\mb{Z})$, using
Lemma \ref{ph14} we have the exact sequences
\begin{equation}\label{phd15}
\cdots \,\longrightarrow\, \Gamma(U_i,\,\Omega^1_X)\,\stackrel{\delta'_i}{\longrightarrow}
\, \Gamma(V_i,\,\Omega^1_X) \,\stackrel{\delta_i}{\longrightarrow}\,
\Gamma(U_i,\, \mc{H}^1_D(\Omega^1_X)) \,\longrightarrow\, \cdots\, .
\end{equation} 
Notice that the sections $\delta_i(df_i/f_i) \,\in\, \Gamma(U_i,\,\mc{H}^1_D(\Omega_X^1))$ agree
on the intersections $U_{ij}\,:=\, U_i\cap U_j$, i.e.,
$$
\delta_i(df_i/f_i)|_{U_{ij}}\,=\,\delta_j(df_j/f_j)|_{U_{ij}}\, .
$$
Indeed, $f_i|_{U_{ij}}\,=\,\lambda_{ij} f_j|_{U_{ij}}$ for
some $\lambda_{ij} \,\in\, \Gamma(U_{ij},\,\mo_{U_{ij}}^\times)$. Then,
\[\left. \frac{df_i}{f_i}\right|_{U_{ij}}\,=\,\left. \frac{d\lambda_{ij}}{\lambda_{ij}}\right|_{U_{ij}}+\left. \frac{df_j}{f_j}\right|_{U_{ij}}.\]
As $\lambda_{ij}$ is invertible, $d\lambda_{ij}/\lambda_{ij}
\,\in\, \Gamma(U_{ij},\Omega^1_X)$. Using the short exact sequence (\ref{phd15}) this
implies that $\delta_j \circ \delta_j'|_{U_{ij}}(d\lambda_{ij}/\lambda_{ij})\,=\,0$.
Hence, $\delta_i(df_i/f_i)|_{U_{ij}}\,=\,\delta_j(df_j/f_j)|_{U_{ij}}$.

Therefore, the local sections $\delta_i(df_i/f_i)\,\in\, \Gamma(U_i,\,
\mc{H}^1_D(\mo_X))$ glue compatibly to define a global section $\{D\}
\,\in\, H^0(X,\, \mc{H}^1_D(\Omega_X^1))$. 

Using the short exact sequence in \eqref{phd15} and arguing as above, it is easy to 
see that the above class $\{D\}$ does not depend on the choice of the 
representatives $f_i$.

\begin{rem}\label{phd17}
Proposition \ref{ph15} implies that 
$H^0(\mc{H}^1_D(\Omega_X^1))\,\cong \, H^1_D(\Omega^1_X)$, while Lemma \ref{ph14} implies 
that we have a homomorphism
$H^1_D(\Omega_X^1) \,\longrightarrow\, H^1(\Omega_X^1)$. The Chern class
$$c_1(\mc{L}) \,\in\, H^1(X,\,\Omega^1_X)$$ is the image of $\{D\}
\,\in\, H^1_D(\Omega_X^1)$ under this homomorphism (see \cite{fgag}).
\end{rem}

We now describe the cup-product $\bigcup c_1(\mc{L})\,:\,H^1(\T_X)
\,\longrightarrow\, H^2(\mo_X)$. Define $U\,:=\,X\backslash D$, and
let $j\,:\,U \,\hookrightarrow\, X$ be the open immersion.
The proof of Proposition \ref{ph15} shows that $\mc{H}^0_D(\mo_X)\,=\,0$. By Lemma
\ref{ph14}, we have a short exact sequence
\begin{equation}\label{br5}
0 \,\longrightarrow\, \mo_X \,\longrightarrow\, j_*\mo_U \,
\stackrel{\delta}{\longrightarrow}\, \mc{H}^1_D(\mo_X) \,\longrightarrow\, 0\, .
\end{equation}

Let \[\lrcorner \{D\}\,:\,\mc{T}_X \,\longrightarrow\, \mc{H}^1_D(\mo_X)\]
be the homomorphism which on each $U_i$ is defined as 
 \[\phi \,\longmapsto\, \delta\left(\frac{\phi(df_i)}{f_i}\right)\, ,\]
where $\delta$ is the projection in \eqref{br5}.
 Proposition \ref{ph15} implies that $H^1(\mc{H}^1_D(\mo_X)) \,\cong\, H^2_D(\mo_X)$. 
 This induces a homomorphism
\[\lrcorner\{D\}\,:\,H^1(\T_X) \,\longrightarrow\,
 H^1(X,\,\mc{H}^1_D(\mo_X)) \,\cong\, H^2_D(\mo_X)\, ;\]
with a slight abuse of notation, this will also be denoted by $\lrcorner\{D\}$. Using Remark
\ref{phd17} one can check that the cup-product map $\cup c_1(\mc{L})$ is the composition
\begin{equation}\label{phd10}
\cup c_1(\mc{L})\,:\,H^1(\T_X) \,\stackrel{\lrcorner\{D\}}{\longrightarrow}\, H^2_D(\mo_X)
\,\longrightarrow\, H^2(\mo_X)\, .
\end{equation}

Let \[\lrcorner \{D\}'\,:\,\N_{D|X} \,\longrightarrow\, \mc{H}^1_D(\mo_X)\]
be the homomorphism which on each $U_i$ is defined as
\[\lrcorner\{D\}'|_{U_i}(\phi) \,=\, \delta\left(\frac{\widetilde{\phi(f_i)}}{f_i}
\right)\, ,\]
where $\widetilde{g}$ for any $g \,\in\, \mo_D(U_i \cap D)$ is its preimage under the
natural surjective homomorphism $\mo_X(U_i) \,\longrightarrow\, \mo_D(U_i \cap D)$.
It follows from the short exact sequence in \eqref{br5} that the map
$\lrcorner\{D\}'$ does not depend on the choice of the lift of $\phi(f_i)$.

\begin{lem}\label{br6}
The above homomorphism $\lrcorner \{D\}'\,:\,\N_{D|X}\,\longrightarrow\,
\mc{H}_D^1(\mo_X)$ is injective.
\end{lem}

\begin{proof}
{}From the short exact sequence in \eqref{br5} it follows that
$\delta\left(\frac{\widetilde{\phi(f_i)}}{f_i}\right)\,=\,0$ on $U_i$ if and only if 
$\widetilde{\phi(f_i)}/f_i \,\in\, \mo_X(U_i)$, which is possible if and only if 
$\widetilde{\phi(f_i)} \,\in \,\mo_X(-D)(U_i)$. But this means that $\phi(f_i)\,=\,0$;
hence $\phi\,=\,0$ because it is determined by its evaluation on $f_i$. Consequently, 
$\lrcorner \{D\}'$ is injective.
\end{proof}

 \begin{cor}
There is a short exact sequence
\begin{equation}\label{br8}
0 \,\longrightarrow\, \N_{D|X}\,\stackrel{\lrcorner \{D\}'}{\longrightarrow}\,
\mc{H}^1_D(\mo_X) \,\stackrel{\psi}{\longrightarrow}\,
\mc{H}^1_D(\mo_X(D)) \,\longrightarrow\, 0\, ,
\end{equation}
where $\psi$ is the homomorphism arising from the natural homomorphism $\mo_X
\,\longrightarrow\, \mo_X(D)$.
 \end{cor}

 \begin{proof}
The injectivity of $\lrcorner \{D\}'$ is proved in Lemma \ref{br6}. Clearly,
$j_*\mo_U \,\cong\, j_*(\mo_X(D)|_U)$. So, we have the following diagram
\[\begin{diagram}
0&\rTo&\mo_X&\rTo&j_*\mo_U&\rTo^{\delta}&\mc{H}^1_D(\mo_X)&\rTo&0\\
& &\dInto&\circlearrowleft&\dTo^{\mr{id}}_{\wr}&\circlearrowleft&\dTo^{\psi}\\
0&\rTo&\mo_X(D)&\rTo&j_*\mo_U&\rTo^{\delta'}&\mc{H}^1_D(\mo_X(D))&\rTo&0 
\end{diagram}\]
where the horizontal short exact sequences are obtained using Lemma \ref{ph14} and
the above identification $j_*\mo_U \,\cong\, j_*(\mo_X(D)|_U)$.
Applying Snake lemma to the above diagram, the homomorphism
$\psi$ is surjective and $\text{kernel}(\psi)$ is isomorphic to the cokernel of 
the homomorphism $\mo_X \,\longrightarrow\, \mo_X(D)$, which is $\N_{D|X}$
by the Poincar\'e adjunction formula. It just
remains to prove that the induced homomorphism from $\N_{D|X}$ to $\mc{H}^1_D(\mo_X)$ 
 is $\lrcorner \{D\}'$ or, equivalently, exactness in the middle of (\ref{br8}).
 
 Using the above diagram, we have $$\ker \psi\,=\,\delta(\ker \delta')
\,=\,\delta(\Ima (\mo_X(D) \to j_*\mo_U))\, .$$ 
Consider the homomorphism $\mo_X(D) \,\longrightarrow\, \N_{D|X}$ defined on open
subsets $U_i$ by 
 $g_i/f_i \,\longmapsto \,\phi$, where $g_i\,\in\, \mo_X(U_i)$ while $\phi$ is defined
as $f_i\,\longmapsto\, g_i \mod \I_D(U_i)$. For such $\phi$, the definition of
$\lrcorner \{D\}'$ states that 
 $\lrcorner \{D\}'|_{U_i}(\phi)\,=\,\delta(g_i/f_i)$. So, for each $g_i/f_i\,\in\,
\mo_X(D)$ we can construct $\phi$ as above such that 
 $\lrcorner \{D\}'|_{U_i}(\phi)\,=\,\delta(g_i/f_i)$. Observe that the induced map from
$\mo_X(D)$ to $\N_{D|X}$ is surjective. Hence, $\ker \psi\,=\,\Ima \lrcorner \{D\}'$.
This completes proof.
\end{proof}

The following lemma will be used in the proof of Proposition \ref{br3} below.
 
\begin{lem}\label{sem2}
Let $i\,:\,D \,\longrightarrow\, X$ be the closed immersion. Then
\[\mc{E}xt^m_X(\mc{H}_D^1(\mo_X(D)),\,i_*\N_{D|X})\,=\,0 \ \ \mbox{ for } m\,=\,0,1\, .\]
 In particular, $\mr{Ext}^1_X(\mc{H}^1_D(\mo_X(D)),\,i_*\N_{D|X})\,=\,0$.
\end{lem}

\begin{proof}
By adjunction,
\[\mc{E}xt^m_X(\mc{H}^1_D(\mo_X(D)),\,i_*\N_{D|X})
\,\cong\, \mc{E}xt^m_D(\mc{H}^1_D(\mo_X(D)) \otimes_{\mo_X} \mo_D,\,\N_{D|X})\, .\]

We claim that $\mc{H}^1_D(\mo_X(D)) \otimes_{\mo_X} \mo_D \,=\,0$.

To prove the claim, using
Lemma \ref{ph14} we have the short exact sequence
 \[0 \,\longrightarrow\, \mo_X(D) \,\longrightarrow\, j_*(\mo_X(D)|_{U})\,
\stackrel{\delta}{\longrightarrow}\, \mc{H}^1_D(\mo_X(D)) \,\longrightarrow\, 0\, .\]
 So, $\mc{H}^1_D(\mo_X(D))$ is supported on $D$. For any $x \,\in\, D$,
$$(j_*\mo_X(D)|_{U})_x \,\cong\, \mo_{X,x}[1/f_x]\, ,$$ where $f_x \,\in\, \mo_{X,x}$ is
the defining equation for $D$ at $x$.
 Any element of $\mc{H}^1_D(\mo_X(D))_x$ is of the form $\delta(g)\,=\,f_x\delta(g/f_x)$,
where $g \,\in\, j_*(\mo_X(D)|_U)_x$. So, $\delta(g) \otimes_{\mo_{X,x}} 1$ is zero in
$\mc{H}^1_D(\mo_X)_x \otimes_{\mo_{X,x}} \mo_{X,x}/(f_x)$, which implies that
$\mc{H}^1_D(\mo_X)_x \otimes_{\mo_{X,x}} \mo_{X,x}/\I_{D,x}\,=\,0$. This proves the claim.

Hence, $\mc{E}xt^m_D(\mc{H}^1_D(\mo_X(D)) \otimes_{\mo_X} \mo_D,\,\N_{D|X})\,=\,0$ for 
 $m\,=\,0,\,1$. Hence, \[\mc{E}xt^m_X(\mc{H}_D^1(\mo_X(D)),\, i_*\N_{D|X})
\,=\,0 \ \ \mbox{ for } m\,=\,0,\,1\, .\]
By Grothendieck Spectral sequence,
\[\mr{Ext}^1_X(\mc{H}^1_D(\mo_X(D)),\,i_*\N_{D|X})
\,\cong\,\bigoplus\limits_{i=0}^1 H^i(\mc{E}xt^{1-i}_X(\mc{H}^1_D(\mo_X(D)),\,
i_*\N_{D|X}))\, .\]
 Hence, $\mr{Ext}^1_X(\mc{H}^1_D(\mo_X(D)),\, i_*\N_{D|X})\,=\,0$, proving the lemma.
 \end{proof}

Given any $t \,\in\, H^1(\T_X)$, denote by $X_t$ the infinitesimal deformation of $X$
along $t$.

\begin{prop}\label{br3}
For any reduced $D \,\in\, |\mc{L}|$, the homomorphism $\cup c_1(\mc{L})$ factors
through $\lrcorner \{D\}$, meaning the following diagram is commutative:
\begin{equation}\label{phd16}
\begin{diagram}
H^1(\mc{T}_X)& \\
\dTo^{\cup c_1(\mc{L})}&\rdTo^{\lrcorner \{D\}} \\
H^2(\mo_X) &\lTo&H^2_D(\mo_X)
\end{diagram}
\end{equation}

Furthermore, given any $t \,\in\, \ker \cup c_1(\mc{L})$ and $X_t$ the corresponding
infinitesimal deformation of $X$, the divisor $D$ lifts to an effective Cartier divisor
in $X_t$ if and only if $t \lrcorner \{D\}\,=\,0$.
 \end{prop}

\begin{proof}
By \cite[Proposition $6.2$]{b1} there is a commutative diagram
\[
\begin{diagram}
H^1(\mc{T}_X)& \\
\dTo^{u^*}&\rdTo^{\lrcorner \{D\}} \\
H^1(\N_{D|X}) &\rTo^{\lrcorner \{D\}'}&H^2_D(\mo_X)
\end{diagram}
\]
where $u^*$ is the composition $H^1(\T_X) \,\longrightarrow\, H^1(\T_X \otimes \mo_D)
\,\longrightarrow\, H^1(\N_{D|X})$; the first homomorphism in this composition is
induced by restriction and the second one by the natural homomorphism $\T_X\otimes
\mo_D\,\longrightarrow\, \N_{D|X}$. The commutativity of \eqref{phd16} then
follows from the definition of cup-product map given in \eqref{phd10}. 
This proves the first part of the proposition.

Using \cite[Proposition 2.6]{b1}, the divisor $D$ lifts to an effective divisor
in $X_t$ if and only if $u^*(t)\,=\,0$;
as before, $X_t$ the infinitesimal deformation of $X$ along $t$. From Lemma \ref{sem2}
it follows that
\[
\mr{Ext}^1_X(\mc{H}_D^1(\mo_X(D)),\,i_*\N_{D|X})\,=\,0\, .
\]
This implies that the short exact sequence in \eqref{br8} splits. Hence, the induced
homomorphism of global sections 
 $H^0(\mc{H}^1_D(\mo_X)) \,\longrightarrow\, H^0(\mc{H}^1_D(\mo_X(D)))$ is surjective. This
means that \[\lrcorner \{D\}'\,:\,H^1(\N_{D|X})\,\longrightarrow\, H^1(\mc{H}^1_D(\mo_X))
\,\cong\, H^2_D(\mo_X)\] is injective.
Hence, $D$ lifts to an effective divisor in $X_t$ if and only if 
$\lrcorner \{D\}(t)\,=\,0$. This completes the proof. 
\end{proof}

We now apply Proposition \ref{br3} to families of smooth projective varieties.

Let $\pi\,:\,\mc{X}\,\longrightarrow\, B$ be a flat family of smooth projective 
varieties with $X$ being the fiber over a base point $o \,\in\, B$.
For any $u \,\in\, B$, denote the fiber $\pi^{-1}(u)$ by $X_u$.

\begin{rem}\label{br10}
The differential of $\pi$ produces a short exact sequence
\[
0 \,\longrightarrow\, \mc{T}_X \,\longrightarrow\, \mc{T}_{\mc{X}}\vert_X\,=\,
\mc{T}_{\mc{X}} \otimes \mo_X
\,\longrightarrow\, \pi^* T_oB \,\longrightarrow\, 0\, ,
\]
where $T_oB$ is the tangent space to $B$
at $o$. The \emph{Kodaira-Spencer map} $$\rho_{\pi}\,:\,T_oB \,\cong\, H^0(\pi^*(T_oB))
\,\longrightarrow\, H^1(\mc{T}_X)$$ is the boundary homomorphism associated to the
above short exact sequence. For an algebraic line bundle $\mc{L}$ on $X$, 
denote by $\NL(\gamma)$ the \emph{Hodge locus} corresponding to 
the Hodge class $$\gamma \,:=\, c_1(\mc{L})
\,\in\, H^{1,1}(X,\,\mb{Q})$$ (see \cite{v5} for definition of Hodge 
locus). Define $\mc{X}'\,:=\,\pi^{-1}(\NL(\gamma))$. After contracting $B$ if
necessary, there is an invertible sheaf 
$\widetilde{\mc{L}}$ on $\mc{X}'$ such that $\widetilde{\mc{L}}|_X \,\cong\, \mc{L}$
(see \cite[\S~3.3.1]{S1}).
\end{rem}

For any $t \,\in\, T_oB$, let $X_t\, \subset\,\mc{X}$ the infinitesimal deformation of $X$
along $t$. For any $u \,\in \,\NL(\gamma)$, define $\widetilde{\mc{L}}_u\,:=\,
\widetilde{\mc{L}}|_{X_u}$.

Given any $b\,\in\, B$ and an invertible sheaf $\mc{L}_b$ on $X_b$, we say that $\mc{L}$ 
\emph{deforms to an invertible sheaf } $\mc{L}_b$ \emph{ on } $X_b$ if there
exists a connected closed subscheme $W \,\subset\, B$ containing both $o$ and $b$, 
and an invertible sheaf $\mc{L}_W$ on $X_W$, such that $\mc{L}_W|_X \,\cong\, \mc{L}$ and
$\mc{L}_W|_{X_b}\,\cong\, \mc{L}_b$.

{}From the upper-semicontinuity theorem for the dimension of global sections it
follows that $\mc{B}_{\mc{L}}$ is a closed subscheme in $\NL(\gamma)$.

Given a family $\pi$ and $\gamma$ as before, the
\emph{Brill-Noether sub-locus of $\NL(\gamma)$ associated to $\mc{L}$}
is the subset $\mc{B}_{\mc{L}}\, \subset\, B$ consisting of
all $b \,\in\, B$ such that there exists a connected closed subscheme
$W \,\subset\, B$ containing both the points $o$ and $b$, and an invertible sheaf
$\mc{L}_W$ on $X_W$, such that
\begin{enumerate}
\item $\mc{L}_W|_X \,\cong\, \mc{L}$, and 

\item $h^0(\mc{L}_W|_{X_w}) \,\geq\, h^0(\mc{L})$ for all $w \,\in\, W$.
\end{enumerate}
{}From Lefschetz ($1,1$)-theorem it follows that $\mc{B}_{\mc{L}}\,\subset\,\NL(\gamma)$.

\begin{prop}\label{br15}
 The tangent space at the point $o \,\in\, \mc{B}_{\mc{L}}$ is
\[T_o\mc{B}_{\mc{L}}\,=\,
\rho_\pi^{-1}\left(\bigcap\limits_{\overset{D \in |\mc{L}|}{\mr{reduced}}} \lrcorner \{D\}\right).\]
\end{prop}

\begin{proof}
For any $t \,\in\, T_oB$, we have $t \,\in\, T_o\NL(\gamma)$ if and only if $\mc{L}$
lifts to an invertible sheaf $\mc{L}_t$ on $X_t$, where $X_t$ is the infinitesimal
deformation of $X$ along $t$ \cite[\S~3.3.1]{S1}. By definition, $t\,\in\,
T_o\mc{B}_{\mc{L}}$ if and only if $t\,\in\, T_o\NL(\gamma)$ and
\begin{equation}\label{eqz1}
\dim_{k[\epsilon]/(\epsilon^2)}
H^0(\mc{L}_t) \,\geq\, h^0(\mc{L})\, .
\end{equation}
By Proposition \ref{br3}, for all reduced divisor $D \,\in\, |\mc{L}|$,
$$\rho_\pi(t) \lrcorner \{D\}\,=\,0$$ if and only if $D$ lifts to an effective Cartier
divisor on $X_t$. Now
this is possible if and only if the natural restriction homomorphism
$H^0(\mc{L}_t) \,\longrightarrow\, H^0(\mc{L})$ is surjective. By the long exact
sequence of cohomologies associated to
\begin{equation}\label{phd11}
0 \,\longrightarrow\, \mc{L} \,\longrightarrow\, \mc{L}_t \xrightarrow{\mbox{ mod t}} \mc{L} \,\longrightarrow\, 0\, ,
\end{equation}
this is equivalent to the statement that
$$\dim_k H^0(\mc{L}_t)\,=\,2h^0(\mc{L})\, .$$ Now, we have
$2\dim_{k[\epsilon]/(\epsilon^2)} H^0(\mc{L}_t)\,=\,\dim_k H^0(\mc{L}_t)$.
 Therefore, for $t \,\in\, T_o\NL(\gamma)$,
the inequality in \eqref{eqz1} holds if and only if
$\rho_\pi(t) \lrcorner \{D\}\,=\,0$ for all reduced $D \,\in\, |\mc{L}|$. This
completes the proof.
\end{proof}

It should be mentioned that it is not true that the
tangent space $T_b\mc{B}_{\mc{L}}$ can be described by a Kodaira-Spencer type formula
$T_bB \,\longrightarrow\, H^1(\T_{X_b})$; it is possible that
 \[T_b\mc{B}_{\mc{L}}\,\not=\, \rho_\pi^{-1}\left(\bigcap\limits_{\overset{D
\in |\mc{L}_b|}{\mr{reduced}}} \lrcorner \{D\}\right)\, .\]
However the following is true.

\begin{cor}\label{br16}
 Suppose that $h^1(\mo_{X_b})\,=\,0$
for all $b \,\in\, \NL(\gamma)$. Then, for any $b\,\in\, \mc{B}_{\mc{L}}$, and a
deformation $\mc{L}_b$ of $\mc{L}$ on $X_b$ with 
$h^0(\mc{L}_b)\,=\,h^0(\mc{L})$ such that a general element of
$|\mc{L}_b|$ is reduced,
\[T_b\mc{B}_{\mc{L}}\,=\,\rho_\pi^{-1}\left(\bigcap\limits_{\overset{D \in |\mc{L}_b|}{\mr{reduced}}} \lrcorner \{D\}\right)\, ,\]
where $\rho_\pi\,:\,T_bB \,\longrightarrow\, H^1(\T_{X_b})$ is the associated
Kodaira-Spencer map.
\end{cor}

\begin{proof}
Since $h^1(\mo_{X_u})\,=\,0$ for all $u \,\in\, \NL(\gamma)$, there is an unique
deformation $\mc{L}_b$ of $\mc{L}$ to an invertible sheaf on $X_b$. By assumption
we have $h^0(\mc{L}_b)\,=\,h^0(\mc{L})$. Since the complete linear system $\mc{L}$ on $\mc{X}_o$ deforms to the complete linear system $H^0(\mc{L}_b)$ in this case, 
observe that $T_b\mc{B}_{\mc{L}}=T_b\mc{B}_{\mc{L}_b}$.
Using the arguments in the proof of Proposition \ref{br15} it follows that
\[T_b\mc{B}_{\mc{L}}\,=\,T_b\mc{B}_{\mc{L}_b}\,=\,\rho_\pi^{-1}\left(\bigcap\limits_{\overset{D \in
|\mc{L}_b|}{\mr{reduced}}} \lrcorner \{D\}\right)\, .\]
This completes the proof.
\end{proof}

\begin{thm}\label{br12}
 If there is an open neighborhood $U \,\subset\, \NL(\gamma)$ of $o$ such that
$h^0(\widetilde{\mc{L}}_u)\,=\,h^0(\mc{L})$ for all $u \,\in\, U$, then $\NL(\gamma) \cap U=\mc{B}_{\mc{L}} \cap U$ and
 $T_u\NL(\gamma)\,=\,T_u\mc{B}_{\mc{L}}$ for all $u \in U$.
\end{thm}

\begin{proof}
By the hypothesis on the theorem, $\NL(\gamma) \cap U=\mc{B}_{\mc{L}} \cap U$.
To prove the statement on the tangent space, note that
$$\dim_{k[\epsilon]/(\epsilon^2)} H^0(\mc{L}_t)\,=\, h^0(\mc{L})$$
for all $t\,\in\, T_o\NL(\gamma)$, because $h^0(\mc{L}_u)\,=\,
h^0(\mc{L})$ for all $u\,\in\, U$. Using the long exact sequence of cohomologies
associated to (\ref{phd11}),
$$
\dim_k H^0(\mc{L}_t) \,\leq\, 2 h^0(\mc{L})
$$
with the equality holding if and only if the
induced homomorphism $H^0(\mc{L}_t) \,\longrightarrow\, H^0(\mc{L})$ is surjective. 
Now, $\dim_k H^0(\mc{L}_t)\,=\,
2\dim_{k[\epsilon]/(\epsilon^2)} H^0(\mc{L}_t)$, and $\dim_{k[\epsilon]/(\epsilon^2)}
H^0(\mc{L}_t) \,\geq\, r$. Hence, 
$H^0(\mc{L}_t) \,\longrightarrow\, H^0(\mc{L})$ is surjective. In other words, every
$D\,\in\, |\mc{L}|$ lifts to an effective Cartier divisor on $X_t$. Then, Proposition
\ref{br3} implies that $t \lrcorner \{D\}\,=\,0$ for any $t \,\in\,
T_o\NL(\gamma)$ and any reduced $D \,\in\, |\mc{L}|$. Now,
$T_o\NL(\gamma) \,\subset\, T_o(\mc{B}_\mc{L})$ by
Proposition \ref{br15}. On the other hand, from the diagram \eqref{phd16} it
follows that the reverse inclusion holds. Hence we conclude that
$T_o\NL(\gamma) \,=\, T_o(\mc{B}_\mc{L})$. Similarly, using the proof of Corollary
\ref{br16}, it can be proved that $T_u\NL(\gamma)\,=\,T_u\mc{B}_{\mc{L}}$
for all $u \,\in \,U$. This completes the proof.
\end{proof}

Observe that when $\dim_o \mc{B}_\mc{L}<\dim_o \NL(\gamma)$, it is not obvious 
$T_o\mc{B}_\mc{L} \,\subsetneqq\, T_o\NL(\gamma)$. In particular, there are
examples of families of smooth projective families $\pi\,:\,\mc{X}\,
\longrightarrow\, B$ and 
an invertible sheaf $\mc{L}$ on $\mc{X}$ such that there exists a point $o \,\in\, B$ for 
which
\begin{enumerate}
\item $h^0(\mc{L}|_{\mc{X}_o})\,>\,h^0(\mc{L}|_{\mc{X}_u})$ for all $u\,\in \,
B \backslash \{o\}$,

\item but $h^0(\mc{L}|_{\mc{X}_t})\,=\,h^0(\mc{L}|_{\mc{X}_o})$
for all \emph{first order} infinitesimal deformation $\mc{X}_t$ of
$\mc{X}_o$, $t \,\in\, T_oB$.
\end{enumerate}
Of course, for any such $\pi$ and invertible sheaf $\mc{L}$, there exists some higher 
order infinitesimal deformation $\mc{X}_{t'}$ of $\mc{X}_o$ for which 
$h^0(\mc{L}|_{\mc{X}_{t'}})\,<\,h^0(\mc{L}|_{\mc{X}_o})$. However, in the next section, we 
produce examples of $\pi$ and $\mc{L}$ such that $(1)$ holds as before and 
furthermore, there exists $t \,\in\, T_oB$ such that 
$h^0(\mc{L}|_{\mc{X}_t})\,<\,h^0(\mc{L}|_{\mc{X}_o})$.

\section{Jumping locus of linear systems}\label{bsec4}

In this section, we produce a family $\pi\,:\,\mc{X} \,\longrightarrow\, B$ and an
invertible sheaf $\mc{M}$ on $\mc{X}$ such that there exists a point $o \,\in\, B$ for
which $$B_{\mc{M}_o} \,\subsetneqq\, \NL(c_1(\mc{M}_o)) \,=\, B$$ and
$T_oB_{\mc{M}_o} \,\subsetneqq\, T_o\NL(\gamma)$. This gives an example of a classical question: Given a 
family of smooth, projective varieties $\pi\,:\,\mc{X} \,\longrightarrow\, B$, when does
there exist a closed fiber $\mc{X}_o$ and an effective divisor $D\,\subset\, \mc{X}_o$
such that there is an infinitesimal deformation $\mc{X}_t$ of $\mc{X}_o$ along
some tangent $t \,\in\, T_oB$ for which the Hodge class $[D] \,\in\, H^{1,1}(\mc{X}_o,
\,\mb{Q})$ lifts to a Hodge class on $\mc{X}_t$ but $D$ does not lift to an
effective Cartier divisor on $\mc{X}_t$?

\begin{se}
Let $Y$ be a smooth projective variety of dimension at least $2$, and let $\mc{L}$
be an invertible sheaf on $Y$. Suppose that the base locus $B$ of $H^0(\mc{L})$ is a (finite)
collection of closed points containing a point $p$ with multiplicity $1$. Define
$B_p\,:=\,B\backslash \{p\}$ and $Z\,:=\,Y\backslash B_p$. 
Consider the closed subscheme $B_p \times Y+\Delta\,\subset\, Y \times Y$, where
$\Delta\,\subset\, Y \times Y$ is the diagonal, and define
\[E_0\,:=\, (B_p \times Y+\Delta) \cap (Y \times Z)\, .\]
\end{se}

\begin{note}
Let $\ov{\pi}\,:\,\ov{Y}\,\longrightarrow\, Y \times Z$ be the blow-up along $E_0$.
The exceptional divisor will be denoted by $E$. Define
$$p_i\,:=\,\pr_i \circ \ov{\pi}\, ,$$ where $\pr_1$ (respectively, $\pr_2$) is
the projection of $Y \times Z$ to $Y$ (respectively, $Z$). 
Let $$\mc{M}\,:=\,p_1^*\mc{L} \otimes \mo_{\ov{Y}}(-E)$$ the invertible sheaf on $\ov{Y}$. 
\end{note}

\begin{lem}\label{br22}
The morphism $p_2\,:\,\ov{Y}\,\longrightarrow\, Z$ is flat. Furthermore, $p_2|_{E}$
is flat.
\end{lem}

\begin{proof}
This is because $E_0$ is flat over $Z$ under the second projection map $\pr_2$.
\end{proof}

\begin{note}
In Remark \ref{br10}, replace the family $\pi\,:\,\mc{X} \,\longrightarrow\, B$ by $p_2$. 
 We have the corresponding Kodaira-Spencer map $\rho_\pi\,:\, T_pY
\,\longrightarrow\, H^1(\T_{\ov{Y}_p})$.
\end{note}

\begin{note}
Denote by $\ov{Y}_q$ (respectively, $E(q)$) the fiber over $q$ under the morphism
$p_2$ (respectively, $p_2|_E$). 
Defile $\mc{M}_q\,:=\,\mc{M} \otimes {\ov{Y}_q}$ and $p_1(q)\,:=\,p_1|_{\ov{Y}_q}$.
Observe that this morphism $p_1(q)$ is surjective as it is simply the blow-up of 
$Y$ along $B_p \cup q$.
\end{note}

\begin{thm}\label{br1}
The inequality $h^0(\mc{M}_y)\,<\,h^0(\mc{M}_p)$ holds for
any $y \,\not=\, p$. In particular, for $\gamma\,=\,c_1(\mc{M}_p)$ on $\ov{Y}_p$,
$$\mc{B}_{\mc{M}_p} \,\not=\, \NL(\gamma)\,=\,Z\, .$$
\end{thm}

\begin{proof}
 Consider the short exact sequence
$$
0 \,\longrightarrow\, \I_{E_0} \,\stackrel{h}{\longrightarrow}\,
\mo_{Y \times Z}\,\longrightarrow\, \mo_{E_0}\,\longrightarrow\, 0\, .
$$
Let
\[ \ov{\pi}^* \I_{E_0} \,\stackrel{\ov{\pi}^*(h)}{\longrightarrow}\,
\mo_{\ov{Y}}\,\longrightarrow\, \mo_E \,\longrightarrow\, 0\]
be its pull back by $\ov{\pi}$. Note that the image of the homomorphism
$\ov{\pi}^*(h)$ is $\I_E$. Hence there is the short exact sequence
\[0 \,\longrightarrow\, \mo_{\ov{Y}}(-E) \,\longrightarrow\, \mo_{\ov{Y}}
\,\longrightarrow\, \mo_E \,\longrightarrow\, 0\, .\]
Tensoring it by $p_1^*\mc{L}$, 
\begin{equation}\label{br21}
 0 \,\longrightarrow\, \mc{M}\,\longrightarrow\, p_1^*\mc{L}
\,\longrightarrow\, p_1^*\mc{L} \otimes \mo_E \,\longrightarrow\, 0\, .
\end{equation}
Now, 
 \[p_1^* \mc{L} \otimes \mo_{\ov{Y}_q} \,\cong\, p_1(q)^* \mc{L}\, , \ \
p_1^* \mc{L} \otimes \mo_E \otimes \mo_{\ov{Y}_q} \,\cong\,
p_1(q)^* \mc{L} \otimes \mo_{E(q)}\, .\]
By Lemma \ref{br22}, the restriction $p_2|_E$ is flat. As $p_1^*\mc{L}$ is an
invertible sheaf on $\ov{Y}$ which is flat over $Z$ via $p_2$, it follows that
$(p_1^*\mc{L} \otimes \mo_E)_y$ is $\mo_{Z,q}$--flat for any $y\,\in \,\ov{Y}_q$. Hence,
$\mr{Tor}^1_{\mo_{Z,q}}((p_1^*\mc{L} \otimes \mo_E)_y,k(q))\,=\,0$
for any $y\,\in\, \ov{Y}_q$.
 Tensoring (\ref{br21}) by $\mo_{\ov{Y}_q}$, yields the short exact sequence
\begin{equation}\label{br23}
0\,\longrightarrow\, \mc{M}_q \,\longrightarrow\, p_1(q)^*\mc{L}
\,\longrightarrow\, p_1(q)^* \mc{L} \otimes \mo_{E(q)}\,\longrightarrow\, 0\, .
\end{equation}
As the morphism $p_1(q)$ is surjective, \[H^0(p_1(q)^* \mc{L} \otimes \mo_{E(q)})
\,=\,H^0(p_1(q)_*(p_1(q)^* \mc{L} \otimes \mo_{E(q)}))\]
and
\[
H^0(p_1(q)^* \mc{L})\,=\,H^0(p_1(q)_*(p_1(q)^* \mc{L}))\, .\]
Define $E_0(q)\,:=\,E_0 \cap Y \times q$. As $p_1(q)$ is the blow-up map of $Y$ along
$B_p \cup q$, it follows that $p_1(q)^*\mo_{E_0(q)} \,\cong\, \mo_{E(q)}$.
 Using the projection formula,
\[H^0(p_1(q)^* \mc{L} \otimes \mo_{E(q)})\,=\,H^0(\mc{L} \otimes \mo_{E_0(q)})\
 \mbox{ and }\ H^0(p_1(q)^* \mc{L})\,=\,H^0(\mc{L})\, .\]
The long exact sequence of cohomologies associated to \eqref{br23} contains
\[0 \,\longrightarrow\, H^0(\mc{M}_q) \,\longrightarrow\, H^0(\mc{L})
\,\stackrel{\rho(q)}{\longrightarrow}\, H^0(\mc{L} \otimes \mo_{E_0(q)})\, ,\]
where $\rho(q)$ is the natural evaluation/restriction map on $E_0(q)$.
By assumption, $\rho(q)\,=\,0$ if and only $q\,=\,p$. So,
$h^0(\mc{M}_q)<h^0(\mc{L})\,=\,h^0(\mc{M}_p)$ for any $q \,\not=\,p$. 

As $\mc{M}_q$ is the restriction of $\mc{M}$ to $\ov{Y}_q$, it follows that $\gamma$
deforms to $c_1(\mc{M}_q)$ because $\ov{Y}_p$ deforms to $\ov{Y}_q$ along the family
$p_2$. Hence, we have $\NL(\gamma)\,=\,Z$. By definition, $\mc{B}_{\mc{M}_p} \,\not=\, Z$.
This completes the proof.
\end{proof}

\begin{rem}\label{br11}
In the proof of Theorem \ref{br1} it was observed that 
$h^0((p_1^*\mc{L})|_{\ov{Y}_q})\,=\,h^0(\mc{L})$ for any closed point
$q \,\in\, Z$. By Grauert's upper semicontinuity theorem \cite[Corollary III.$12.9$]{R1},
this implies that for any ring homomorphism $\phi\,:\,\mo_{Y,p} \,\longrightarrow\,
k(p)[\epsilon]/(\epsilon^2)$, the homomorphism
\[
H^0(p_1^*\mc{L}) \otimes k(p)[\epsilon]/(\epsilon^2)
\,\longrightarrow\, H^0(p_1^*\mc{L} \otimes_\phi k(p)[\epsilon]/(\epsilon^2))
\] is an isomorphism. Since $h^0(p_1^*\mc{L})\,=\,h^0(p_{1_*}p_1^*\mc{L})
\,=\,h^0(\mc{L})$, this implies that
\[\dim_{k(p)[\epsilon]/(\epsilon^2))} H^0(p_1^*\mc{L} \otimes_\phi k(p)
[\epsilon]/(\epsilon^2))\,=\,h^0(\mc{L})\, .\]
\end{rem}

\begin{note}[Restriction to infinitesimal deformation]
 Fix a point $q \,\in\, Z$ and a tangent $t \,\in\, T_qZ$ corresponding to a ring
homomorphism $\phi\,:\,\mo_{Z,q} \,\longrightarrow\, k(q)[\epsilon]/(\epsilon^2)$.
Denote by $\ov{Y}_t$ the 
 infinitesimal deformation of $\ov{Y}_q$ along $t$, so $\ov{Y}_t$ is the fiber
product $\ov{Y} \times_Z \Spec k(q)[\epsilon]/(\epsilon^2)$ with respect to the morphism 
 $\Spec k(q)[\epsilon]/(\epsilon^2) \,\longrightarrow\, Z$ induced by $\phi$. Given any
sheaf $\mc{F}$ on $\ov{Y}$, denote by $\mc{F}_t$ the pull-back of $\mc{F}$ to $\ov{Y}_t$.
 In particular,
$$\mc{F}_t \,\cong\, \mc{F} \otimes_\phi k(q)[\epsilon]/(\epsilon^2)\, ,$$
where the sheaf $\mc{F} \otimes_\phi k(q)[\epsilon]/(\epsilon^2)$ is defined by
 \[(\mc{F} \otimes_\phi k(q)[\epsilon]/(\epsilon^2))_{y}
\,=\,\mc{F}_{y} \otimes_{\phi} k(q)[\epsilon]/(\epsilon^2)\]
 if $y \,\in\, \ov{Y}_q$ and zero otherwise (consider $k(q)[\epsilon]/(\epsilon^2)$ as
a constant sheaf supported on $\ov{Y}_q$).
 \end{note}
 
\begin{thm}\label{br14}
 There exist $t \,\in \,T_pZ$ and $D \,\in \,H^0(\mc{M}_p)$ such that
$$\rho_\pi(t) \,\not=\,0\, ,\ \ t \cup c_1(\mc{M}_p)\,=\,0$$ but
$t \lrcorner \{D\} \,\not=\, 0$. In particular, 
$t \,\in\, T_p\NL(c_1(\mc{M}_p))\,=\,T_pZ$, but $t \,\not\in\, T_p\mc{B}_{\mc{M}_p}$.
\end{thm}

\begin{proof}
 Given a morphism $\phi\,:\,\mo_{Z,p} \,\longrightarrow\,
k(p)[\epsilon]/(\epsilon^2)$, the following exact sequence is obtained by
tensoring (\ref{br21}) with $- \otimes_{\phi} k(p)[\epsilon]/(\epsilon^2)$:
\[0 \,\longrightarrow\, H^0(\mc{M} \otimes_\phi k(p)[\epsilon]/(\epsilon^2))
\,\longrightarrow\, H^0(p_1^*\mc{L} \otimes_\phi k(p)[\epsilon]/(\epsilon^2))\]
\[\stackrel{\rho(\phi)}{\longrightarrow}\, H^0(p_1^*\mc{L} \otimes
\mo_E \otimes_\phi k(p)[\epsilon]/(\epsilon^2))\, .\]
 It suffices to show that there exists a morphism $\phi$ such that
$\rho(\phi)$ is not the zero map. Indeed, given such a $\phi$, by Remark \ref{br11},
 \[\dim_{k(p)[\epsilon]/(\epsilon^2)} H^0(\mc{M} \otimes_\phi
k(p)[\epsilon]/(\epsilon^2))\,<\, \dim_{k(p)[\epsilon]/(\epsilon^2)}
H^0(p_1^*\mc{L} \otimes_\phi k(p)[\epsilon]/(\epsilon^2))\,=\,h^0(\mc{L})\, .\]
 In the proof of Theorem \ref{br1}
it was observed that $h^0(\mc{L})\,=\,h^0(\mc{M}_p)$ when
$\rho(p)\,=\,0$. This implies that for $t \,\in\, T_pZ$ corresponding to $\phi$, we have
$\rho(t) \,\not=\, 0$ and $t \,\not\in\, T_p\mc{B}_{\mc{M}_p}$. 
 By Proposition \ref{br15}, there exists $D \,\in\, H^0(\mc{M}_p)$ such that
$t \cup c_1(\mc{M}_p)\,=\,0$ but $t \lrcorner \{D\} \,\not=\, 0$. This will prove
the theorem.
 
As $p$ is a reduced base point of $H^0(\mc{L})$, there exist $s \,\in\, H^0(\mc{L})$,
$f_s \,\in\, m_p\backslash m_p^2$ and $g_s \,\in\, \mc{L}_p$ such that 
$s_p\,=\,f_sg_s$ and $g_s \,\not\in\, m_p\mc{L}_p$; here $s_p$ is the image of $s$
under the localization morphism $H^0(\mc{L}) \,\longrightarrow\, \mc{L}_p$. 
By assumption, $Y$ is smooth. Since $f_s \,\in\, m_p\backslash m_p^2$, we can choose
a regular sequence $(f_s,f_1,\cdots ,f_m)$ generating the maximal ideal $m_p$. 
Let $$\phi\,:\,\mo_{Z,p}
\,\longrightarrow\, k(p)[\epsilon]/(\epsilon^2)$$
be the ring homomorphism defined by
$1 \,\longmapsto\, 1$, $f_s \,\longmapsto\, \epsilon$ and $f_i \,\longmapsto\, 0$ for all
$i\,=\,1,\cdots ,m$. Then, $s$ defines a non-zero element
$$s \otimes 1 \,\in\, H^0(p_1^*\mc{L} \otimes_\phi k(p)[\epsilon]/(\epsilon^2))$$
and the image of $s \otimes 1 \,\in\, H^0(p_1^*\mc{L} \otimes_\phi k(p)[\epsilon]/(\epsilon^2))$ under the 
natural homomorphism \[\rho(\phi)'_p\,:\,H^0(p_1^*\mc{L} \otimes_\phi k(p)[
\epsilon]/(\epsilon^2)) \,\longrightarrow\, H^0(\mc{L}_p \otimes_\phi k(p)[\epsilon]/(\epsilon^2))\] is non-zero.
 
Now, $H^0(p_1^*\mc{L} \otimes \mo_E \otimes_\phi k(p)[\epsilon]/(\epsilon^2))
\,=\,H^0(\ov{\pi}^* (\pr_1^*\mc{L}) \otimes \ov{\pi}^*\mo_{E_0}
\otimes_\phi k(p)[\epsilon]/(\epsilon^2))$,
which by the projection formula is equal to
\[H^0(\pr_1^*\mc{L} \otimes \mo_{E_0} \otimes_\phi k(p)[\epsilon]/(\epsilon^2))\,=\,
\bigoplus\limits_{q \in B} H^0((\pr_1^*\mc{L} \otimes \mo_{E_0})_{q \times p}
\otimes_\phi k(p)[\epsilon]/(\epsilon^2))\, .\]
Observe that $(\pr_1^*\mc{L})_{q \times p} \,\cong\, \mc{L}_q$. Recall
that the composition $\Delta \,\hookrightarrow\, Y \times Y \,
\stackrel{\pr_1}{\longrightarrow}\,Y$
is an isomorphism, hence $\pr_1^{\#}\,:\,\mo_{Y,p} 
\,\stackrel{\sim}{\longrightarrow}\,
\mo_{\Delta,p \times p}$. Since the only irreducible component of $E_0$ containing
$p \times p$ is $\Delta$, we have 
 \[(\pr_1^*\mc{L} \otimes \mo_{E_0})_{p \times p} \otimes_\phi k(p)[\epsilon]/
(\epsilon^2) \,\cong\, \mc{L}_p \otimes_{\pr_1^{\#}} \mo_{\Delta, p \times p}
\otimes_\phi k(p)[\epsilon]/(\epsilon^2)\]
 \[\cong\, \mc{L}_p \otimes_{\pr_1^{\#}} \mo_{Y,p}
\otimes_\phi k(p)[\epsilon]/(\epsilon^2)
\,\cong\, \mc{L}_p \otimes_\phi k(p)[\epsilon]/(\epsilon^2)\, , \]
 where $M \otimes_{\pr_1^{\#}} N$ is a tensor product of
$\mo_{Y \times Z, p \times p}$--modules viewed as $\mo_{Y,p}$--modules under the
morphism $\pr_1^{\#}$.
Write the evaluation map $\rho(\phi)\,=\,\oplus_{q \in B} \rho(\phi)_q$, where $\rho(\phi)_q$ is the restriction of the evaluation map to $q$. 
Then $\rho(\phi)_p$ coincides with the morphism $\rho(\phi)'_p$ defined above. Since $\rho(\phi)'_p$ is non-zero, so is $\rho(\phi)$.
This completes the proof of the theorem.
\end{proof}
 
 \section{Applications to curve counting}
 
\begin{se}
Let $Y$ be a smooth projective surface and $\mc{L}$ an invertible sheaf on $Y$. Denote by
$r\,:=\,h^0(\mc{L})$. Fix $m$ distinct points $p_1,\cdots ,p_m$ on $Y$. 
Define \[W\,:=\,\{s \,\in\, H^0(\mc{L})\,\mid\, s(p_i)\,=\,0 \ \ \forall\ 
i\,=\,1,\cdots ,m\}\]
and
\[W_q\,:=\,\{s \,\in\, H^0(\mc{L})\,\mid\, s(p_i)\,=\,0\ \ \forall\ i\,=\,
2,\cdots ,m,\ \mbox{ and }\, s(q)\,=\,0\}\, .\]
For $r\,>\,0$, define
\[\mc{Z}_{\mc{L}}^r(p_1,\cdots,p_m)\,:=\,\{q \,\in\, Y\backslash \{p_2, p_3, ...,p_m\}
\,\mid\, \dim W_q \,\ge\, r\}\, .\]
 \end{se}

 Observe that $\mc{Z}_{\mc{L}}^r(p_1,\cdots ,p_m)$ is the set of points
$q \,\in\, Y\backslash \{p_2,\cdots ,p_m\}$ such that there exists at least an $r$ dimensional 
 family of curves in the linear system $|\mc{L}|$ which passes through the points
$q, p_2, p_3,\cdots ,p_m$. We will prove that the locus of such points
coincides with the Brill-Noether type locus defined in the previous section.
 
\begin{thm}\label{br18}
Notations as in Section \ref{bsec4}. Substitute $p\,=\,p_1$ and $B
\,=\,\{p_1,\cdots ,p_m\}$. For $r\,=\,\dim W$, we have
$\mc{Z}_{\mc{L}}^r(p_1,\cdots ,p_m)\,=\,\mc{B}_{\mc{M}_{p_1}}$.
\end{thm}

\begin{proof}
By the proof of Theorem \ref{br1}, there is a short exact sequence
\[0 \,\longrightarrow\, H^0(\mc{M}_q)\,\longrightarrow\, H^0(\mc{L})
\, \stackrel{\rho(q)}{\longrightarrow}\,
H^0(\mc{L} \otimes \mo_{E_0(q)})\, .\]
Recall that $E_0(q)\,=\,B_p \cup q$ and $\rho(q)$ is the evaluation at $E_0(q)$. Then by
definition, $\ker \rho(q)\,=\,W_q$. Hence, $H^0(\mc{M}_q)\,=\,W_q$ and
$H^0(\mc{M}_{p_1})\,=\,W$. The theorem now follows directly. 
\end{proof}

\section*{Acknowledgements}

We would like to thank Dr. Nicola Tarasca for helpful discussions. IB acknowledges 
support of a J. C. Bose Fellowship.

\end{document}